\documentclass{amsart}
\usepackage[T1]{fontenc}
\usepackage{mathtools}
\usepackage{amsmath,amscd}
\usepackage{amssymb}
\usepackage{amsthm}

\usepackage{comment}
\usepackage{enumitem}
\usepackage{graphicx}

\usepackage{mathrsfs}
\usepackage[ocgcolorlinks, linkcolor=blue]{hyperref}

\usepackage[
backend=biber,
style=alphabetic,doi=false,isbn=false,url=false,minalphanames=5,maxalphanames=10,giveninits=true,maxbibnames=99
]{biblatex}
\usepackage{calc}
{\begin{list}{\arabic{enumi}.}{\usecounter{enumi}%
			\setlength{\labelsep}{0.5em}%
			\settowidth{\labelwidth}{(\arabic{enumi})}%
	\setlength{\leftmargin}{\labelwidth+\labelsep}}}%
{\end{list}}

\usepackage{comment}

\DeclareRobustCommand{\SkipTocEntry}[5]{}

\usepackage{xcolor}
\definecolor{LOcolor}{RGB}{150,100,0}

\newtheorem{Theorem}{Theorem}[section]

\newtheorem{Lemma}[Theorem]{Lemma}

\newtheorem{Proposition}[Theorem]{Proposition}

\theoremstyle{definition}

\newtheorem{Remark}[Theorem]{Remark}

\numberwithin{equation}{section}

\newcommand{\mR}{\mathbb{R}}                    % Formatting for R
\DeclarePairedDelimiterX{\seminorm}[1]{[}{]}{#1}
\DeclarePairedDelimiterX{\norm}[1]{\lVert}{\rVert}{#1}
\DeclarePairedDelimiterX{\inner}[2]{\langle}{\rangle}{\,#1,\,#2}
\DeclarePairedDelimiterX{\abs}[1]{\lvert}{\rvert}{#1}

\newcommand{\ol}[1]{\overline{#1}}

\newcommand{\eps}{\varepsilon}

\newcommand{\p}{\partial}

\newcommand{\e}{\varepsilon}
\newcommand{\Om}{\Omega}

\newcounter{sidenote}

\begin{document}

\title{Inverse problems for semilinear elliptic equations with low regularity} 

\author[D. Johansson]{David Johansson}
\address{Department of Mathematics, Aarhus University}
\email{johansson@math.au.dk}

\author[J. Nurminen]{Janne Nurminen}
\address{Computational Engineering, School of Engineering Sciences, Lappeenranta-Lahti University of Technology, Finland \& Department of Mathematics and Statistics, University of Jyväskylä, Jyväskylä, Finland}
\email{janne.s.nurminen@jyu.fi, janne.nurminen@lut.fi}

\author[M. Salo]{Mikko Salo}
\address{Department of Mathematics and Statistics, University of Jyäskylä}
\email{mikko.j.salo@jyu.fi}

%\subjclass{53C25, 53C21, 58F17, 35J15}

%\date{\today}

%\maketitle

\begin{abstract}
We show that a general nonlinearity $a(x,u)$ is uniquely determined, possibly up to a gauge, in a neighborhood of a fixed solution from boundary measurements of the corresponding semilinear equation. The main theorems are low regularity counterparts of the results in our recent paper (Johansson, Nurminen, Salo; ArXiv preprint 2312.12196).
\end{abstract}

\maketitle

\section{Introduction}\label{sec_intro}

This article deals with inverse boundary value problems for semilinear equations of the form 
\[
\Delta u(x) + a(x,u(x)) = 0 \text{ in $\Omega$},
\]
where $\Omega  \subseteq  \mR^n$ is a bounded open set with smooth boundary and the function $a(x,t)$ represents the nonlinearity. There is a large literature on this topic. The first order linearization method introduced in \cite{Isakov1993} has been employed to show that nonlinearities $a(x,t)$ satisfying conditions such as 
\begin{align*}
a(x,0) &= 0, \\
\p_t a(x,t) &\leq 0,
\end{align*}
can be determined in a certain reachable set from boundary measurements. The first condition above ensures that $0$ is a solution, and the second sign condition guarantees that a maximum principle and well-posedness hold. Various results of this type, also for somewhat more general nonlinearities, may be found in \cites{IS, IN, Sun, ImanuvilovYamamoto2013}. See also the surveys \cites{Sun2005, Uhlmann2009}.  In the results based on first order linearization, one typically uses known results on inverse problems for the linearized equation.

On the other hand, the higher order linearization method introduced in \cites{FO, LLLS} (following the hyperbolic case in \cite{KLU}) applies to nonlinearities that do not need to satisfy any sign condition. Moreover, this method uses the nonlinearity as a beneficial tool and yields results in certain nonlinear cases where the corresponding results for linear equations remain unsolved \cites{KU, LLLS2, CFO23, FKO23, ST, Nurminen_unbounded}. However, the method might only allow one to determine the Taylor series of $a(x,t)$ at $t=0$, and it does not in general determine $a(x,t)$ for $t \neq 0$. The higher order linearization method has also been used outside of semilinear equations. For quasilinear equations one has the similar phenomenon of determining only the Taylor series of the unknown coefficient \cites{CFKKU, KKU}. Also inverse problems for the minimal surface equation (which is an example of a quasilinear equation) on Riemannian manifolds have been studied with this method. In \cites{nurminenMSE, nurminenMSE2} the Taylor series of a conformal factor for metrics in the same conformal class is recovered. In dimension two it is shown in \cites{CLLO, CLLT, CLT} (in slightly different settings) that three linearizations are enough to determine the Riemannian metric up to an isometry. 

Recently in \cite{JNS2023} (see also \cite{nurminensahoo} for a similar result for a biharmonic operator with a second order nonlinearity) we gave a result showing that from boundary measurements near the zero solution, one can determine a general nonlinearity $a(x,t)$ near $t=0$ whenever $a(x,0) = 0$. We also gave a similar result without the assumption $a(x,0) = 0$, but in that case $a(x,t)$ can only be determined up to a natural gauge transformation. The precise assumption for the nonlinearity was that $a(x,t)$ should be $C^{1,\alpha}$ in $x$ and $C^3$ in $t$. In this article we improve the regularity assumptions to $L^{r}$ in $x$ and $C^{1,\alpha}$ in $t$, where $r > n/2$. We also simplify the proofs in the process. The method is based on first linearization and a Runge approximation argument.

Let us state the main results. We will assume that $\Omega  \subseteq  \mR^n$, $n \geq 2$, is a bounded open set with smooth boundary (though many arguments would remain valid for $C^{1,1}$ boundaries). We consider nonlinearities $a \in L^r(\Om, C^{1,\alpha}(\mR))$ for some $r > n/2$. If $u \in L^{\infty}(\Omega)$ with $M = \norm{u}_{L^{\infty}(\Om)}$, then $a(x,u(x)) \in L^r(\Omega)$ since 
\[
\int_{\Omega} |a(x,u(x))|^r \,dx \leq \int_{\Omega} \sup_{|t| \leq M} |a(x,t)|^r \,dx \leq \norm{a}_{L^r(\Omega, L^{\infty}([-M,M]))}^r.
\]
We denote by $W^{k,p}(\Om)$ the standard Sobolev spaces in $\Om$. If $u \in W^{2,r}(\Om)$ with $r > n/2$, then (after choosing a suitable representative) $u \in C^{\alpha}(\ol{\Om})$ for some $\alpha > 0$ by Sobolev embedding, so $a(x,u(x))$ is well defined.

We wish to study inverse problems without any further assumptions on the nonlinearity $a(x,t)$. In particular we do not assume well-posedness of the Dirichlet problem, and hence the boundary measurements will be formulated in terms of Cauchy data sets (see \cite{FSU_book}). Given a solution $w \in W^{2,r}(\Om)$ of $\Delta w + a(x,w) = 0$ in $\Om$, and given $\delta > 0$, we define the Cauchy data set for solutions near $w$ as 
\[
C_{a}^{w,\delta} = \{ (u|_{\p \Om}, \p_{\nu} u|_{\p \Om}) \,:\, u \in W^{2,r}(\Om), \ \ \Delta u + a(x,u) = 0 \text{ in $\Om$}, \ \ \norm{u-w}_{W^{2,r}(\Om)} \leq \delta \}.
\]
If $s > 1/r$ we denote the trace space of $W^{s,r}(\Om)$ by $W^{s-\frac{1}{r},r}(\p \Om) := B^{s-\frac{1}{r}}_{rr}(\p \Om)$, see \cite{Triebel1983}. Then $C_a^{w,\delta}$ is a subset of $W^{2-\frac{1}{r},r}(\p \Om) \times W^{1-\frac{1}{r},r}(\p \Om)$. If the semilinear equation happens to be well-posed for Dirichlet data close to $w|_{\p \Om}$, then the set $C_{a}^{w,\delta}$ contains the graph of the corresponding nonlinear Dirichlet-to-Neumann map for Dirichlet data close to $w|_{\p \Om}$ and vice versa. For further discussion on the set $C_{a}^{w,\delta}$ and its relation to Dirichlet-to-Neumann maps, we refer the reader to \cite{JNS2023}.

Our first result shows that if two nonlinearities $a_1$ and $a_2$ admit a common solution $w$ and have the same Cauchy data for solutions near $w$, then $a_1(x,t) = a_2(x,t)$ for $t$ near $w(x)$. We assume that the nonlinearities are $C^{1,\alpha}$ in $t$, but in fact uniform $C^1$ regularity in $t$ would be enough (see Remark \ref{remark_c1_uniform}).

\begin{Theorem} \label{thm_main1}
Let $a_1, a_2 \in L^r(\Om, C^{1,\alpha}(\mR))$ where $r > n/2$, $r \geq 2$ and $\alpha > 0$. Suppose that $w \in W^{2,r}(\Om)$ solves $\Delta w + a_j(x,w) = 0$ in $\Om$ for $j = 1,2$. If for some $\delta, C > 0$ one has 
\[
C_{a_1}^{w,\delta}  \subseteq  C_{a_2}^{0,C},
\]
then there is $\eps > 0$ such that 
\[
a_1(x,w(x)+\lambda) = a_2(x,w(x)+\lambda) \text{ for a.e.\ $x \in \Omega$ and for all $\lambda \in (-\eps, \eps)$.}
\]
\end{Theorem}

In particular, if $a_j(x,0) = 0$, we may take $w \equiv 0$ as the common solution and the conclusion is that $a_1(x,\lambda) = a_2(x,\lambda)$ for small $\lambda$. The additional assumption $r \geq 2$ is only needed in dimensions $n \in \{ 2,3 \}$, since for $n \geq 4$ it follows from the condition $r > n/2$.

The second result concerns two general nonlinearities $a_j(x,t)$ that may not have a common solution. In this case one can only determine the nonlinearity up to gauge (see \cite{Sun}). The gauge is given by 
\[
(T_{\varphi} a)(x,t) = \Delta \varphi(x) + a(x,t+\varphi(x))
\]
where $\varphi \in W^{2,r}(\Om)$ satisfies $\varphi|_{\p \Om} = \p_{\nu} \varphi|_{\p \Om} = 0$.

\begin{Theorem} \label{thm_main2}
Let $a_1, a_2 \in L^r(\Om, C^{1,\alpha}(\mR))$ where $r > n/2$, $r \geq 2$ and $\alpha > 0$. Suppose that $w_1 \in W^{2,r}(\Om)$ solves $\Delta w_1 + a_1(x,w_1) = 0$ in $\Om$. If for some $\delta, C > 0$ one has 
\[
C_{a_1}^{w_1,\delta}  \subseteq  C_{a_2}^{0,C},
\]
then there exist $\varphi \in W^{2,r}(\Om)$ with $\varphi|_{\p \Om} = \p_{\nu} \varphi|_{\p \Om} = 0$ and $\eps > 0$ such that 
\[
a_1(x,w_1(x)+\lambda) = (T_{\varphi }a_2)(x,w_1(x)+\lambda) \text{ for a.e.\ $x \in \Omega$ and for all $\lambda \in (-\eps, \eps)$.}
\]
\end{Theorem}

In Theorem \ref{thm_main2}, $\varphi = w_1-w_2$ where $w_2$ is the unique solution of $\Delta w_2 + a_2(x,w_2) = 0$ in $\Om$ that has the same Cauchy data as $w_1$ (see Lemma \ref{lem_carleman_wtwor}). If $w_2=w_1$, we obtain Theorem \ref{thm_main1} as a special case.

We note that our results cover the linear case $a(x,u) = q(x)u$, and hence include the recovery of an $L^r(\Omega)$ potential where $r > n/2$ and $r \geq 2$. This misses the endpoint result $r=n/2$ for $n \geq 3$, see \cites{Chanillo1990, Nachman1992, DKS2013}, and also the result for $L^{4/3+\eps}$ potentials for $n=2$ \cite{Blåsten2020}. The exponent $r=n/2$ for $n \geq 3$ may be considered optimal in the context of standard wellposedness theory for the Dirichlet problem with $L^p$ potentials, and also for the strong unique continuation principle to hold \cite{JerisonKenig1985}. Another low regularity result for $n=2$, under a condition on $\p_t a(x,t)$ that ensures well-posedness, is given in \cite{IN}.

We also remark that it is not in general possible to determine $a(x,t)$ in $\Omega \times \mR$, see \cite{IS, FKO23}. In general one could expect to determine $a(x,t)$ in the reachable set $\{ (x, u(x)) \,:\, x \in \Omega, \ \Delta u + a(x,u) = 0 \text{ in $\Om$} \}$. Our methods imply that the reachable set is always an open set. Conditions for existence of solutions $u$ may be found in \cite[Section 14.1 and Exercise 8]{TaylorPDE3}.

\subsection*{Methods}

The main results are low regularity versions of the corresponding results in \cite{JNS2023}, and we will prove them by a method based on first linearization. The linearization of the equation at a solution $w$ is 
\[
\Delta v + qv = 0 \text{ in $\Om$}
\]
where $q(x) = \p_u a(x,w(x))$. Since we do not assume well-posedness, there may be a finite dimensional obstruction for solving the Dirichlet problem for this equation. In \cite{JNS2023} we employed a solvability result where the Dirichlet data was modified by a function in $\p_{\nu} N_q$ where $N_q$ is the eigenspace corresponding to zero eigenvalue. The reason for requiring that $a(x,t)$ is $C^{1,\alpha}$ in $x$ in \cite{JNS2023} was that $\p_{\nu} N_q$ is not in the natural trace space of solutions if $a$ has low regularity. In this article we replace the space $\p_{\nu} N_q$ with another space $D_q$ of the same dimension. This makes it possible to work with nonlinearities $a(x,t)$ that have only $L^r$ regularity in $x$.

After establishing solvability for the linearized equation, the next step is to construct a $C^1$ map that maps a small solution $v$ of the linearized equation $\Delta v + qv = 0$ to a solution $u = S(v) = S_{a,w}(v)$ of the nonlinear equation $\Delta u + a(x,u) = 0$, so that 
\[
u = w + v + o(\norm{v})
\]
where $w$ is a fixed solution of $\Delta w + a(x,w) = 0$. In \cite{JNS2023} the existence of $S$ was shown via a Banach fixed point argument, and the fact that $S$ is $C^1$ was proved by using the implicit function theorem. In this work we establish both the existence and smoothness of $S$ by a single application of the implicit function theorem, which leads to a much shorter argument.

Next we consider the setting of Theorem \ref{thm_main1}  and employ the Cauchy data set inclusion $C_{a_1}^{w,\delta}  \subseteq  C_{a_2}^{0,C}$ to define a map $T_{a_2,w}$, which takes a solution $v$ of the linearized equation $\Delta v + \p_u a(x,w(x))v = 0$ to a solution $u$ of $\Delta u + a_2(x,u) = 0$ that has the same Cauchy data as $S_{a_1,w}(v)$. Since $T_{a_2,w}$ is defined via the Cauchy data inclusion, we do not know if it is $C^1$. However, by another argument based on the implicit function theorem we prove that $T_{a_2,w}$ is indeed $C^1$. The argument involves a unique continuation result that leads to Lemma \ref{lem_carleman_wtwor}, and a certain projection operator defined in Lemma \ref{lemma_projection_operator} via a fourth order equation. Lemma \ref{lemma_projection_operator} is the only place where we need the additional condition $r \geq 2$ for $n \in \{2,3\}$.

After having proved that both maps $S_{a_1,w}$ and $T_{a_2,w}$ are $C^1$, the argument for recovering the nonlinearity proceeds as in \cite{JNS2023}. We first derive an integral identity involving the difference of potentials in the linearized equations. Then we invoke the completeness of products of solutions to linear Schr\"odinger equations \cites{Chanillo1990, Nachman1992, DKS2013, Blåsten2020} and a unique continuation argument to show that 
\[
\p_u a_1(x, S_{a_1,w}(v)) = \p_u a_2(x, S_{a_1,w}(v))
\]
for small solutions $v$ of the linearized equation for $a_1$. It remains to show that there is $\eps$ such that for any $x_0 \in \Om$ and any $\lambda \in (-\eps,\eps)$, one can find $v$ such that $S_{a_1,w}(v)(x_0) = \lambda$.  This follows from a Runge approximation argument for the linearized equation with $L^r$ potentials.

This article is structured as follows. Section \ref{sec_intro} is the introduction, Section \ref{sec_solvability_linear} studies solvability and regularity for the linearized equation, and Section \ref{sec_first_sol_map} presents the solution map $S_{a,w}$. The second solution map $T_{a_2,w}$ is studied in Section \ref{sec_second_sol_map_new}. Finally, the main theorems are proved in Section \ref{sec_proof_of_main} and the required Runge approximation result in Section \ref{sec_runge}.

\subsection*{Acknowledgements}

The authors are partly supported by the Research Council of Finland (Centre of Excellence in Inverse Modelling and Imaging and FAME Flagship, grants 353091 and 359208). J.N.\ is also supported by the Research Council of Finland (Flagship of Advanced Mathematics for Sensing Imaging and Modelling grant 359183) and by the Emil Aaltonen Foundation.

\section{Solvability for linear equations} \label{sec_solvability_linear}

Throughout this article we assume that $\Omega$ is a bounded connected open subset of $\mR^n$, $n \geq 2$, with smooth boundary. Let $q\in L^{r}(\Om)$ where $r > n/2$, and define 
\begin{equation*}
    N_q:=\{\psi\in H^1_0(\Om): (\Delta+q)\psi=0\}.
\end{equation*}
%In the lemma below, if $q \in C^{1,\alpha}(\ol{\Om})$ one can choose $D_q = \p_{\nu} N_q$.

\begin{Proposition}  \label{prop_linearized_ri_wonep}
Let $q \in L^{r}(\Om)$ where $r > n/2$. There is a subspace $D_q = \mathrm{span}\{h_1, \ldots, h_m\}$ of $W^{2-1/r,r}(\p \Om)$ with $\dim(D_q) = \dim(N_q) < \infty$ such that for any $F \in L^r(\Om)$ and $f \in W^{2-1/r,r}(\p \Om)$, there is a unique function $\Phi = \Phi(F,f) \in D_q$ such that the problem 
\begin{equation}\label{eq_schrodinger1}
\begin{cases}
                \Delta u + qu = F\quad&\text{in }\Omega, \\

                u = f + \Phi \quad&\text{on }\partial\Omega,
\end{cases}
\end{equation}
admits a solution $u \in W^{2,r}(\Om)$. If $\{ \psi_1, \ldots, \psi_m \}$ is a suitable basis of $N_q$, the function $\Phi$ is given by
\begin{equation}\label{eq_finite_rank_map}
                \Phi(F, f) = \sum_{j=1}^m \left(\int_{\Om}F\psi_j\,dx + \int_{\partial\Om} f\partial_{\nu}\psi_j\,dS\right) h_j.
\end{equation}

Moreover, there is unique solution $u_{F,f} = G_q(F,f)$ such that $u_{F,f} \perp N_q$ where $\perp$ means $L^2$-orthogonal. The solution $u_{F,f}$ depends linearly on $F$ and $f$ and satisfies
\begin{equation}\label{uff_estimate}
\norm{u_{F,f}}_{W^{2,r}(\Om)} \leq C(\norm{F}_{L^r(\Om)} + \norm{f}_{W^{2-1/r,r}(\p \Om)}),
\end{equation}
where $C$ is independent of $F$ and $f$.
\end{Proposition}

We begin with a standard regularity result. We will write $L^{p+} := \bigcup_{t > p} L^t$. In this article we will frequently use the Sobolev embedding 
\[
W^{s,p}(\Om) \subseteq \left\{ \begin{array}{cc} L^{\frac{np}{n-sp}}(\Om), & s < n/p, \\[3pt] \bigcap_{t < \infty} L^t(\Om), & s = n/p, \\[3pt] C^{1-n/p}(\ol{\Om}), & s > n/p, \end{array} \right.
\]
as well as the generalized H\"older inequality 
\[
\int uvw \,dx \leq \norm{u}_{L^{p_1}} \norm{v}_{L^{p_2}} \norm{w}_{L^{p_3}}, \qquad \frac{1}{p_1} + \frac{1}{p_2} + \frac{1}{p_3} = 1.
\]

\begin{Lemma}  \label{lemma_regularity_linear}
Let $q \in L^{r}(\Om)$ where $r > n/2$. Then any $u \in H^1(\Omega)$ solving 
\begin{equation*}
\begin{cases}
                \Delta u + qu = F\quad&\text{in }\Omega, \\

                u = f \quad&\text{on }\partial\Omega,
\end{cases}
\end{equation*}
where $F \in L^r(\Omega)$ and $f \in W^{2-1/r,r}(\p \Omega)$, must satisfy $u \in W^{2,r}(\Om)$. Moreover, $N_q  \subseteq  W^{2,r}(\Om)$ is finite dimensional.
\end{Lemma}
\begin{proof}
By using the right inverse of the trace, this reduces to showing that any $w \in H^1_0(\Omega)$ solving 
\begin{equation*}
\begin{cases}
                \Delta w + qw = G\quad&\text{in }\Omega, \\

                w = 0 \quad&\text{on }\partial\Omega,
\end{cases}
\end{equation*}
where $G \in L^r(\Omega)$, must satisfy $w \in W^{2,r}(\Omega)$. Rewrite the equation as 
\[
\Delta w = G - qw.
\]
If $n=2$, then by Sobolev embedding $w \in W^{1,2}  \subseteq  L^t$ for any $t < \infty$ and therefore $qw \in L^{n/2+}$. Thus $G-qw \in L^{n/2+}$, so $w \in W^{2,n/2+}$ by \cite[Theorem 0.3]{JerisonKenig}. By Sobolev embedding again $w \in L^{\infty}$, so $G - qw \in L^r$ and thus $w \in W^{2,r}$ by \cite[Theorem 0.3]{JerisonKenig}. The proof for $n \geq 3$ is similar, but the initial Sobolev embedding $W^{1,2}  \subseteq  L^{\frac{2n}{n-2}}$ only gives $w \in W^{1,2+\delta}$ for some $\delta > 0$. However, bootstrapping this regularity argument finitely many times yields $w \in W^{1,t}$ for some $t > n$, and then the argument above gives $w \in W^{2,r}$. We omit the details.

The finite dimensionality of $N_q$ is standard when $q \in L^{\infty}$. Under the current assumption $q \in L^r$, if $\psi \in N_q$, then the weak formulation of the equation implies 
\[
\int_{\Om} |\nabla \psi|^2 \,dx = \int_{\Om} q \psi^2 \,dx \leq \norm{q}_{L^r} \norm{\psi}_{L^t}^2
\]
where $1/r+1/t+1/t = 1$, i.e.\ $t = 2r' < \frac{2n}{n-2}$. If we equip $N_q$ with the $L^t(\Omega)$ norm, then any bounded sequence in $N_q$ is also bounded in the $H^1$ norm and by compact Sobolev embedding has a convergent subsequence in the $L^t$ norm since $t < \frac{2n}{n-2}$. Thus the identity map on $N_q$ is compact, which implies that $N_q$ must be finite dimensional.
\end{proof}

\begin{proof}[Proof of Proposition \ref{prop_linearized_ri_wonep}]
Let $\{ \psi_1, \ldots, \psi_m \}$ be an orthonormal basis of $N_q$ with respect to the inner product in $L^2(\Om)$. Given $F \in L^r(\Om)$ and $f \in W^{2-1/r,r}(\p \Om)$, we first solve
\begin{equation}\label{eq_schrödinger21}
\begin{cases}
                \Delta v + qv = F - \sum a_k \psi_k \quad&\text{in }\Omega, \\

                v = f \quad&\text{on }\partial\Omega,
\end{cases}
\end{equation}
for suitable $a_j \in \mR$. The compatibility conditions are obtained by integrating the equation against $\psi_j$, i.e.
\[
a_j = a_j(F,f) = \int_{\Om} F \psi_j \,dx + \int_{\p \Om} f \p_{\nu} \psi_j \,dS.
\]
With this choice of $a_j$, we first find a weak solution $v \in H^1(\Om)$ of \eqref{eq_schrödinger21} (note that $f \in W^{2-1/r,r}(\p \Om) \subseteq H^{1/2}(\p \Om)$ by Sobolev embedding). If $q \in L^{\infty}(\Om)$ this is standard and follows from \cite[Theorem 8.6]{GT}. If $q \in L^r(\Om)$ one first reduces the problem to the case $f=0$ by using a right inverse of the trace. If $\lambda > 0$ is sufficiently large, the operator $T: H^1_0(\Om) \to H^{-1}(\Om), Tv = (\Delta + q - \lambda)v$ is an isomorphism (see e.g.\ \cite[Appendix A]{DKS2013}), and the problem becomes equivalent with 
\[
v + \lambda T^{-1} j(v) = T^{-1}(F - \sum a_k \psi_k)
\]
where $j: H^1_0(\Om) \to H^{-1}(\Om)$ is the natural inclusion. Since $j$ is compact, the Fredholm alternative holds for the above equation, and there is a solution $v$ if and only if $T^{-1}(F - \sum a_k \psi_k)$ is orthogonal to the kernel of $(\mathrm{Id} + \lambda T^{-1} j)^*$. But this is equivalent with $F - \sum a_k \psi_k$ being $L^2$-orthogonal to $N_q$, which yields the existence of a weak solution $v \in H^1(\Om)$. Lemma \ref{lemma_regularity_linear} then shows that $v \in W^{2,r}(\Om)$. 

We have now proved that with the given choice of $a_j$, \eqref{eq_schrödinger21} has a solution $v \in W^{2,r}(\Om)$. We next choose $u_j \in W^{2,r}(\Om)$ solving $(\Delta+q)u_j = \psi_j$, $1 \leq j \leq m$, and set  $D_q = \mathrm{span}\{ u_1|_{\p \Om}, \ldots, u_m|_{\p \Om} \}$. To see that such $u_j$ exist, it is enough to choose $B_r$ to be a large ball with $\ol{\Om}  \subseteq  B_r$ such that $0$ is not a Dirichlet eigenvalue of $\Delta+q$ in $B_r$, and to solve 
\begin{equation*} %\label{eq_schrödinger2}
\begin{cases}
                \Delta \tilde{u}_j + \tilde{q} \tilde{u}_j = \tilde{\psi}_j \quad&\text{in }B_r, \\

                \tilde{u}_j = 0 \quad&\text{on }\partial B_r,
\end{cases}
\end{equation*}
where $\tilde{q}$ and $\tilde{\psi}_j$ are the extensions of $q$ and $\psi_j$ by zero to $B_r$. The fact that one can arrange $0$ not to be a Dirichlet eigenvalue follows from strict monotonicity of Dirichlet eigenvalues with respect to increasing the domain, see the argument in \cite{Leis1967}. The solutions $\tilde{u}_j$ are in $W^{2,r}(B_r)$ by Lemma \ref{lemma_regularity_linear}, and we can take $u_j = \tilde{u}_j|_{\Om}$.

We also show that $\dim(D_q) = m$: if $\sum c_j u_j|_{\p \Om} = 0$, then $u = \sum c_j u_j$ satisfies $(\Delta+q)u = \psi$ with $u|_{\p \Om} = 0$ where $\psi = \sum c_j \psi_j$. Integrating the equation against $\psi$ gives $\int \psi^2 \,dx = 0$, so $\psi = 0$ and hence $c_1 = \ldots = c_m = 0$ since $\psi_j$ are linearly independent. Thus $\dim(D_q) = m$.

Now that we have obtained the functions $u_j$, we obtain a solution $u$ to \eqref{eq_schrodinger1} via
\[
u = v + \sum a_j u_j.
\]
Moreover, any solution to \eqref{eq_schrodinger1} is of the form $u + \psi$ for some $\psi \in N_q$.

Next we show that the function $\Phi$ is unique. Assume we have two solutions $u_1, u_2$ to \eqref{eq_schrodinger1} with the same data $F,f$ but with $\Phi_1, \Phi_2$ respectively. Then $u = u_1-u_2$ solves 
\begin{equation*}
\begin{cases}
                \Delta u + qu = 0\quad&\text{in }\Omega, \\

                u = \Phi_1-\Phi_2 = \sum c_j u_j|_{\p \Om} \quad&\text{on }\partial\Omega,
\end{cases}
\end{equation*}
for some $c_j$. Writing $v = u - \sum c_j u_j$, we see that $v \in W^{2,r}(\Om)$ solves $(\Delta+q)v = \psi$ and $v|_{\p \Om} = 0$ where $\psi = -\sum c_j \psi_j \in N_q$. Integrating this equation against $\psi$ yields $\int \psi^2 \,dx = 0$, so $c_j = 0$ for all $j$ and therefore $\Phi_1 = \Phi_2$.

Using the uniqueness of $\Phi$ we can show that the $N_q$-orthogonal solution to \eqref{eq_schrodinger1} obtained above is unique. Assume that we have two solutions $u_1,u_2$ with the same data $F,f$ and with $u_j \perp N_q$. Then $u_1-u_2 \in N_q$ and $u_1-u_2 \perp N_q$, which implies $u_1 = u_2$. Finally, let 
\[
X = \{ u \in W^{2,r}(\Om) \,:\, u \perp N_q \}, \qquad Y = \{ f \in W^{2-1/r,r}(\p \Om) \,:\, f \perp D_q \}
\]
and consider the map 
\[
T: X \to L^r(\Om) \times Y, \ \ Tu = (\Delta u + qu, (\mathrm{Id}- P_{D_q})(u|_{\p \Om}))
\]
where $P_{D_q}$ the $L^2(\p \Om)$-orthonormal projection to $D_q$. By what we have proved above, $T$ is a bounded linear bijective operator. The open mapping theorem ensures that $T$ has a bounded inverse. This implies that $u_{F,f}$ depends continuously on $F$ and $f$.
\end{proof}

\section{Solution map for nonlinear equation}\label{sec_first_sol_map}

The following result shows the existence of a map $v \mapsto S(v)$ that parametrizes solutions $u$ of the semilinear equation $\Delta u + a(x,u) = 0$ in $\Om$, when $u$ is close to a fixed solution $w$, in terms of solutions of the corresponding linearized equation.

\begin{Proposition} \label{prop_first_solution_map}
Let \( a\in L^r(\Om, C^{1,\alpha}(\mR)) \) where $r > n/2$, and suppose that \( w \in W^{2,r}(\Om) \) is a solution of
\begin{equation*}
	\Delta w + a(x,w) = 0 \text{ in }\Omega.
\end{equation*}
Let $q(x) = \p_u a(x,w(x))$. There is a \( C^{1} \) map \( S = S_{a,w}: V \to W^{2,r}(\Om) \), where $V$ is a neighborhood of $0$ in $W^{2,r}(\Om)$, such that \( u = S(v) \) solves 
\[
\Delta u + a(x,u) = \Delta v + qv.
\]
One has \( S(0) = w \), $(DS)_0 = \mathrm{Id}$, $S(v)-w-v \perp N_q$ and $S(v)-w-v|_{\p \Om} \in D_q$.

Conversely, any solution of $\Delta u + a(x,u) = 0$ with $\norm{u-w}_{W^{2,r}(\Om)}$ small enough must be of the form $u = S(v)$ for some $v \in W^{2,r}(\Om)$ solving $\Delta v + qv = 0$. The function $v$ is given by
\begin{equation*}
    v=P_{N_q}(u-w) + \tilde v
\end{equation*}
where $P_{N_q}$ is the $L^2(\Om)$-orthogonal projection to $N_q$, and $\tilde v=G_q(0,(u-w)|_{\p\Om})$ is the unique solution given by Proposition \ref{prop_linearized_ri_wonep}.
\end{Proposition}

Note that the result gives the Taylor expansion 
\[
S(v) = w + v + R(v)
\]
where $\norm{R(v)}_{W^{2,r}(\Om)} = o(\norm{v}_{W^{2,r}(\Om)})$ as $\norm{v}_{W^{2,r}(\Om)} \to 0$ and $R(v) \in N_q^{\perp}$, $R(v)|_{\p \Om} \in D_q$.

For the proof, we give a lemma on the properties of the map $u \mapsto a(x,u)$.

\begin{Lemma} \label{lemma_axu_differentiable_map}
Let $a \in L^r(\Om, C^{1,\alpha}(\mR))$ where $r > n/2$ and $\alpha > 0$. Then $u \mapsto a(\,\cdot\,, u(\,\cdot\,))$ is a $C^1$ map from $W^{2,r}(\Om)$ to $L^r(\Om)$.
\end{Lemma}
\begin{proof}
To prove that $u \mapsto a(\,\cdot\,, u(\,\cdot\,))$ is $C^1$, we need to show that for any $u, v \in W^{2,r}(\Om)$ one has 
\[
\norm{a(x,u+v)-a(x,u)-\p_u a(x,u) v}_{L^r} = o(\norm{v}_{W^{2,r}}) \quad \text{as $\norm{v}_{W^{2,r}} \to 0$.}
\]
Let $u, v \in W^{2,r}(\Om)$ with $\norm{u}_{L^{\infty}} \leq M$ and $\norm{v}_{L^{\infty}} = \eps \leq 1$. We obtain 
\begin{align*}
\norm{a(x,u+v)-a(x,u)-\p_u a(x,u)v}_{L^r}^r &= \int_{\Om} |a(x,u(x)+v(x))-a(x,u(x))-\p_u a(x,u(x))v(x)|^r \,dx \\ &\leq \int_{\Om} \sup_{t,h} |a(\,\cdot\,, t+h) - a(\,\cdot\,,t) - \p_u a(\,\cdot\,,t)h|^r \,dx
\end{align*}
where the supremum is over $|t| \leq M$ and $|h| \leq \eps$. Since $a \in L^r(\Om, C^{1,\alpha}(\mR))$, we have 
\[
\int_{\Om} \sup_{|t|,|s| \leq M+1} \frac{|\p_u a(x,t) - \p_u a(x,s)|^r}{|t-s|^{\alpha r}} \,dx \leq C_M < \infty.
\]
In particular, for supremum over $|t| \leq M$, $|h| \leq \eps$ we have 
\begin{align*}
\int_{\Om} \sup_{t,h} |a(x, t+h) - a(x,t) - \p_u a(x,t)h|^r \,dx &= \int_{\Om} \sup_{t,h} |h \int_0^1 (\p_u a(x, t + sh) - \p_u a(x, t)) \,ds|^r \,dx \\
 &\leq \eps^r \int_{\Om} \sup_{t,h} \sup_{s \in [0,1]} |\p_u a(x, t + sh) - \p_u a(x, t)|^r \,dx \\
 &= \eps^r \int_{\Om} \sup_{t,h} \sup_{s \in [0,1]} \frac{|\p_u a(x, t + sh) - \p_u a(x, t)|^r}{|sh|^{r\alpha}} |sh|^{r\alpha} \,dx \\
 &\leq \eps^{r+r\alpha} \int_{\Om} \sup_{t,h} \sup_{s \in [0,1]} \frac{|\p_u a(x, t + sh) - \p_u a(x, t)|^r}{|sh|^{r\alpha}} \,dx \\
 &\leq C_M \eps^{r+r\alpha}.
\end{align*}
Combining these estimates gives 
\begin{align*}
\norm{a(x,u+v)-a(x,u)-\p_u a(x,u)v}_{L^r} \leq C_M \norm{v}_{L^{\infty}}^{1+\alpha} \leq C_M \norm{v}_{W^{2,r}}^{1+\alpha}.
\end{align*}
This shows that the map $u \mapsto a(x,u)$ is $C^1$ as required.
\end{proof}

\begin{Remark} \label{remark_c1_uniform}
Lemma \ref{lemma_axu_differentiable_map} is valid with the same proof when $a \in L^r(\Om, C^{1,\eta_T}([-T,T]))$ for any $T > 0$, where $C^{1,\eta_T}([-T,T])$ is the space of functions in $C^1([-T,T])$ whose first derivatives have a fixed modulus of continuity $\eta_T$. The point is that this modulus of continuity should be uniform over $x \in \Om$.
\end{Remark}

\begin{proof}[Proof of Proposition \ref{prop_first_solution_map}]
Let $D_q$ be as in Proposition \ref{prop_linearized_ri_wonep}, and define the map 
\[ 
F\colon W^{2,r}(\Omega)\times W^{2,r}(\Omega) \to L^r(\Omega) \times (W^{2-1/r,r}(\p \Om) \cap D_q^{\perp}) \times N_{q}
\]
by
\begin{equation*}
	F(u,v)\coloneqq (\Delta u+a(x,u) - \Delta v - q v, P_{D_q^{\perp}}((u-w-v)\vert_{\partial\Omega}), P_{N_{q}}(u-w-v)).
\end{equation*}
Since $a \in L^r(\Om, C^{1,\alpha}(\mR))$ where $r > n/2$, the map $F$ is $C^1$, by Lemma \ref{lemma_axu_differentiable_map}, and satisfies $F(w,0) = (0,0,0)$. Its Fr\'echet derivative with respect to $u$ satisfies 
\begin{equation} \label{duf_formula}
(D_u F)_{(w,0)}(h) = (\Delta h + q h, P_{D_q^{\perp}}(h|_{\p \Om}), P_{N_q} h).
\end{equation}
Then $(D_u F)_{(w,0)}$ is a bounded linear operator $ W^{2,r}(\Omega) \to L^r(\Omega)\times (W^{2-1/r,r}(\p \Om) \cap D_q^{\perp}) \times N_{q}$. It is bijective by Proposition \ref{prop_linearized_ri_wonep}, so the open mapping theorem ensures that it is an isomorphism. Now the implicit function theorem in Banach spaces implies the existence of a $C^1$ map $S: V \to U$, where $U$ and $V$ are open sets in $W^{2,r}(\Om)$ with $0 \in V$ and $w \in U$, such that $F(S(v),v) = 0$ for $v \in V$ and the solution $S(v)$ is unique in the sense that 
\[
F(u,v) = 0 \text{ for $(u,v) \in U \times V$} \quad \implies \quad u = S(v).
\]
%The implicit function theorem also gives the uniqueness of such solution close to $w$.
One has $S(0) = w$, and differentiating $F(S(v),v) = 0$ with respect to $v$ gives 
\[
(D_u F)_{(w,0)} (DS)_0(\tilde{v}) - (\Delta \tilde{v} + q \tilde{v}, P_{D_q^{\perp}}(\tilde{v}|_{\p \Om}), P_{N_q}(\tilde{v})) = 0.
\]
From \eqref{duf_formula} we obtain 
\[
(D_u F)_{(w,0)} ( (DS)_0(\tilde{v}) - \tilde{v} ) = 0.
\]
The fact that $(D_u F)_{(w,0)}$ is an isomorphism gives that $(DS)_0(\tilde{v}) = \tilde{v}$. The identity $F(S(v),v) = 0$ yields $S(v)-w-v \perp N_q$ and $S(v)-w-v|_{\p \Om} \in D_q$.

Conversely, let $u \in W^{2,r}(\Om)$ solve $\Delta u + a(x,u) = 0$. Apply Proposition \ref{prop_linearized_ri_wonep} to find the unique function $v \in W^{2,r}(\Om)$ such that 
\[
\Delta v + qv = 0, \qquad v|_{\p \Om} \in (u-w)|_{\p \Om} + D_q, \qquad P_{N_q} v = P_{N_q}(u-w).
\]
Then $v = P_{N_q}(u-w) + G_q(0, u-w|_{\p \Om})$ and $F(u,v) = 0$. If $\norm{u-w}_{W^{2,r}}$ is sufficiently small, the uniqueness notion above implies that $u = S(v)$.
\end{proof}

For the next result, we define $q=\p_ua(x,w),$ $q_v=\p_ua(x,S_{a,w}(v))$ and
\begin{align*}
    V_{\tilde q}= \{h\in W^{2,r}(\Om) : \Delta h + \tilde q h= 0\}.
\end{align*}

\begin{Lemma}\label{S_deriv_isom}
    In the setting of Proposition \ref{prop_first_solution_map} let $v\in V_q$ be small. Then $DS_{a,w}(v)\colon V_q\to V_{q_v}$ is an isomorphism.
\end{Lemma}

\begin{proof}
    The proof is the same as in \cite[Lemma 2.5]{JNS2023}, by noting that $R(v)$ after Proposition \ref{prop_first_solution_map} is $C^1$ since $S_{a,w}(v)$ is $C^1$ and that $DR(0)=0$.
\end{proof}

\section{Second solution map}\label{sec_second_sol_map_new}

In order to prove our main result, we need to construct solutions to the equations $\Delta u + a_j(x,u) = 0$ such that these solutions have the same Cauchy data and they both depend smoothly on some linear solution $v$. If we attempt to use only the solution maps $S_{a_1,w_1}$ and $S_{a_2,w_2}$, then the latter condition requires that solutions of $(\Delta + \p_ua_1(x,w_1))v=0$ also solve $(\Delta + \p_ua_1(x,w_2))v=0$. This we do not know a priori. Moreover, there is no guarantee that even the Dirichlet data of $S_{a_1,w_1}(v)$ and $S_{a_2,w_2}(v)$ are equal, let alone the Cauchy data. Thus we need to replace $S_{a_2,w_2}$ with another solution map $T_{a_2,w_1}$ that uses information on the Cauchy data sets. Next we construct $T_{a_2,w_1}$, and in Proposition \ref{prop_second_solution_map} we prove that it is $C^1$ in $v$ without requiring that $\p_ua_1(x,w_1)=\p_ua_1(x,w_2)$.

We now assume that $a_1, a_2 \in L^r(\Om, C^{1,\alpha}(\mR))$ for some $r > n/2$ are two nonlinearities, $w_1 \in W^{2,r}(\Om)$ solves $\Delta w_1 + a_1(x,w_1) = 0$ in $\Om$, and that one has the local Cauchy data set inclusion 
\[
C_{a_1}^{w_1,\delta}  \subseteq  C_{a_2}^{0,C}.
\]
In particular, this implies that there is $w_2 \in W^{2,r}(\Om)$ solving $\Delta w_2 + a_2(x,w_2) = 0$ and having the same Cauchy data as $w_1$ on $\p \Om$.

Let $S_{a_1,w_1}$ be the $C^1$ solution map for the nonlinearity $a_1$ given in Proposition \ref{prop_first_solution_map}. If we write $q_1 = \p_u a_1(x, w_1(x))$ and 
\[
V_{q_1,\delta_1} = \{ v \in W^{2,r}(\Om) \,:\, \Delta v + q_1 v = 0, \quad \norm{v}_{W^{2,r}} < \delta_1 \},
\]
then for $v \in V_{q_1,\delta_1}$ with $\delta_1 > 0$ small enough, $u = S_{a_1,w_1}(v)$ solves $\Delta u + a(x,u) = 0$ with expansion 
\[
u = w_1 + v + o(\norm{v}_{W^{2,r}})
\]
as $v \to 0$. From the inclusion $C_{a_1}^{w_1,\delta}  \subseteq  C_{a_2}^{0,C}$ we obtain another map 
\[
T_{a_2,w_1}: V_{q_1,\delta_1} \to W^{2,r}(\Om), \ \ v \mapsto u_2,
\]
where $u_2 \in W^{2,r}(\Om)$ is a solution of $\Delta u_2 + a_2(x,u_2) = 0$ obtained from the inclusion $C_{a_1}^{w_1,\delta}  \subseteq  C_{a_2}^{0,C}$ and having the same Cauchy data as $u_1 = S_{a_1,w_1}(v)$ on $\p \Om$.

The solution map $T_{a_2,w_1}$ above produces solutions of $\Delta u + a_2(x,u) = 0$ parametrized by solutions $v$ of $\Delta v + q_1 v = 0$, but we do not know yet if this map is $C^1$ with respect to $v$. We will prove this next.

\begin{Proposition} \label{prop_second_solution_map}
Let  $a_1, a_2 \in L^r(\Om, C^{1,\alpha}(\mR))$ where $r > n/2$ and $r \geq 2$, and suppose that $C_{a_1}^{w_1,\delta}  \subseteq  C_{a_2}^{0,C}$. The map $T_{a_2,w_1}$ above is a $C^1$ map $V_{q_1,\delta_1} \to W^{2,r}(\Om)$ when $\delta_1 > 0$ is small enough. Each $u = T_{a_2,w_1}(v)$ solves $\Delta u + a_2(x,u) = 0$ and has the same Cauchy data on $\p \Om$ as $S_{a_1,w_1}(v)$.
\end{Proposition}

For the proof, we need a few lemmas that are low regularity counterparts of corresponding results in \cite[Section 3]{JNS2023} with constants only depending on an upper bound on $\norm{q}_{L^r(\Om)}$. The first lemma is a quantitative version of the elliptic regularity result in Lemma \ref{lemma_regularity_linear}.

\begin{Lemma} \label{lemma_cauchy_data_1}
Let $\Omega  \subseteq  \mR^{n}$ be a bounded open set with $C^{\infty}$ boundary, let $r>n/2$ and $M > 0$. There is $C > 0$ depending on $M$ such that for any $q \in L^{r}(\Om)$ with $\norm{q}_{L^r} \leq M$ and for any $u \in W^{2,r}(\Om)$ we have 
\[
    \norm{u}_{W^{2,r}(\Om)} \leq C (\norm{(\Delta+q)u}_{L^{r}(\Om)} + \norm{u|_{\p \Om}}_{W^{2-1/r,r}(\p\Om)} + \norm{u}_{H^1(\Om)}).
\]
\end{Lemma}
\begin{proof}
    First consider the case $q=0$. We look at the Banach space $X = L^{r}(\Om) \times W^{2-1/r,r}(\p\Om) \times H^1(\Om)$, with norm 
    \[
    \norm{(F,f,v)}_X = \norm{F}_{L^{r}(\Om)} + \norm{f}_{W^{2-1/r,r}(\p\Om)} +   \norm{v}_{H^1(\Om)},
    \]
    and the bounded, linear, injective map
    \[
    T\colon W^{2,r}(\Om) \to X, \ \ T(u) = (\Delta u, u|_{\p \Om}, j(u)),
    \]
    where $j: W^{2,r}(\Om) \to H^1(\Om)$ is the natural inclusion (here we use Sobolev embedding). 
    Then $T$ has closed range. To see this, suppose that $u_j \in W^{2,r}(\Om)$ and $T(u_j) \to (F, f, v)$ in $X$. Then $u_j \to v$ in $H^1(\Om)$, $u_j|_{\p \Om} \to f$ in $W^{2-1/r,r}(\Om)$ and $\Delta u_j \to F$ in $L^{r}(\Om)$. On the other hand $\Delta u_j \to \Delta v$ in $H^{-1}(\Om)$ and $u_j|_{\p \Om} \to v|_{\p \Om}$ in $H^{1/2}(\p \Om)$, and by uniqueness of limits one has $(\Delta+q)v = F$ and $v|_{\p \Om} = f$. By Lemma \ref{lemma_regularity_linear} the weak solution $v$ satisfies $v \in W^{2,r}(\Om)$. Thus $(F,f,v) = T(v)$ and $\mathrm{Ran}(T)$ is closed.

    We have proved that $T: W^{2,r}(\Om) \to \mathrm{Ran}(T)$ is a bounded linear bijection between Banach spaces. By the open mapping theorem it has a bounded inverse $S: \mathrm{Ran}(T) \to W^{2,r}(\Om)$, and thus for any $u \in W^{2,r}(\Om)$ one has 
    \[
    \norm{u}_{W^{2,r}(\Om)} = \norm{STu}_{W^{2,r}(\Om)} \leq C \norm{Tu}_{X}.
    \]
    This proves the claim for $q=0$.

    Next we assume that $\norm{q}_{L^r} \leq M$. Given $u \in W^{2,r}(\Om)$, the result proved above for $q=0$ yields 
    \begin{equation} \label{u_wtwor_intermediate}
    \norm{u}_{W^{2,r}(\Om)} \leq C (\norm{(\Delta+q)u}_{L^{r}(\Om)} + \norm{qu}_{L^r(\Om)} + \norm{u|_{\p \Om}}_{W^{2-1/r,r}(\p\Om)} + \norm{u}_{H^1(\Om)}).
    \end{equation}
    By Sobolev embedding, 
    \begin{align*}
    \norm{qu}_{L^{r}(\Om)} \leq M \norm{u}_{L^{\infty}(\Om)} \leq C M \norm{u}_{W^{s_\theta,p_\theta}(\Om)}
\end{align*}
for some sufficiently small $\theta > 0$, where 
\[
s_\theta = (1-\theta) \cdot 2 + \theta \cdot 1, \qquad \frac{1}{p_\theta} = (1-\theta) \frac{1}{r} + \theta \frac{1}{2}.
\]
Now using interpolation we get
\begin{equation*}
    \norm{u}_{W^{s_\theta,p_\theta}(\Om)}\leq \norm{u}_{W^{1,2}(\Om)}^{\theta}\norm{u}_{W^{2,r}(\Om)}^{1-\theta}.
\end{equation*}
Then using Young's inequality with $\e$ gives  
\begin{equation*}
    \norm{q u}_{L^{r}(\Om)}\leq C M \norm{u}_{H^1(\Om)}^{\theta}\norm{u}_{W^{2,r}(\Om)}^{1-\theta} \leq \e\norm{u}_{W^{2,r}(\Om)} + C_{\e,M}\norm{u}_{H^1(\Om)}.
\end{equation*}
Choosing $\e$ small enough we can absorb the $\e\norm{u}_{W^{2,r}(\Om)}$ term to the left hand side of \eqref{u_wtwor_intermediate}. This proves the result.
\end{proof}

The next lemma gives an estimate for $\norm{u}_{W^{2,r}(\Om)}$ in terms of $\norm{(\Delta+q)u}_{L^r(\Om)}$ and the Cauchy data of $u$. The proof invokes a unique continuation result for $L^r$ potentials.

\begin{Lemma} \label{lemma_linear_cd_estimate}
Let $r > n/2$ and $M > 0$. There is $C > 0$ depending on $M$ such that for any $q \in L^r(\Om)$ with $\norm{q}_{L^r} \leq M$ and  for any $u \in W^{2,r}(\Om)$, one has 
\[
\norm{u}_{W^{2,r}(\Om)} \leq C(\norm{(\Delta+q)u}_{L^r(\Om)} + \norm{u|_{\p \Om}}_{W^{2-1/r,r}(\p \Om)} + \norm{\p_{\nu} u|_{\p \Om}}_{W^{1-1/r,r}(\p \Om)}).
\]
\end{Lemma}
\begin{proof}
We argue by contradiction and assume that for any $m$ there exist $q_m$ with $\norm{q_m}_{L^r} \leq M$ and $u_m \in W^{2,r}$ such that 
\begin{equation} \label{eq_contradiction}
    \norm{u_m}_{W^{2,r}(\Om)} > m (\norm{(\Delta + q_m) u_m}_{L^{r}(\Om)} + \norm{u_m}_{W^{2-1/r,r}(\p \Om)} + \norm{\p_{\nu} u_m}_{W^{1-1/r,r}(\p \Om)}).
\end{equation}
On the other hand, Lemma \ref{lemma_cauchy_data_1} implies that 
\[
\norm{u_m}_{W^{2,r}(\Om)} \leq C (\norm{(\Delta+q_m)u_m}_{L^{r}(\Om)} + \norm{u_m}_{W^{2-1/r,r}(\p\Om)} + \norm{u_m}_{H^1(\Om)}).
\]
Normalize $u_m$ so that $\norm{u_m}_{H^1(\Om)} = 1$. Then using \eqref{eq_contradiction} yields 
\begin{align*}
    \norm{u_m}_{W^{2,r}(\Om)} \leq C(\frac{1}{m}\norm{u_m}_{W^{2,r}(\Om)} + 1).
\end{align*}
Then $\norm{u_m}_{W^{2,r}(\Om)} \leq C$ uniformly when $m$ is sufficiently large. Using that $W^{2,r}(\Om)$ is a reflexive Sobolev space, and by using compact Sobolev embedding, for any $\eps > 0$ there is a subsequence, still denoted by $(u_m)$, such that 
\begin{align*}
u_m \rightharpoonup u & \text{ weakly in $W^{2,r}(\Om)$}, \\
u_m \to u &\text{ in $W^{2-\eps,r}(\Om)$}.
\end{align*}
Passing to a further subsequence, we may further assume that 
\begin{align*}
q_m \rightharpoonup q & \text{ weakly in $L^r(\Om)$}, \\
q_m \to q & \text{ in $W^{-\eps,r}(\Om)$}.
\end{align*}

On the other hand, from \eqref{eq_contradiction} and the bound $\norm{u_m}_{W^{2,r}(\Om)} \leq C$ we see that 
\[
u_m|_{\p \Om} \to 0, \qquad \p_{\nu} u_m|_{\p \Om} \to 0, \qquad (\Delta+q_m)u_m \to 0
\]
in the respective spaces. Now uniqueness of limits implies that $u|_{\p \Omega} = 0$ and $\p_{\nu} u|_{\p \Om} = 0$. Here we chose $\eps > 0$ so that $2-\eps > 1-1/r$, which ensures that $\p_{\nu} u|_{\p \Om}$ is well defined by the trace theorem. We also claim that 
\[
(\Delta+q_m)u_m \to (\Delta+q)u \text{ in $W^{-\eps,r}(\Om)$.}
\]
This follows since $q_m u_m \to qu$ by the estimate $\norm{ab}_{W^{-\eps,r}} \leq C \norm{a}_{W^{-\eps,r}} \norm{b}_{C^{\eps+\delta}} \leq C \norm{a}_{W^{-\eps,r}} \norm{b}_{W^{2-\eps,r}}$ (see \cite[Theorem 3.3.2]{Triebel1983}), which holds for some $\delta > 0$ by Sobolev embedding when $\eps > 0$ is chosen really small using that $r > n/2$. Now by uniqueness of limits, we see that $u \in W^{2,r}(\Om)$ is a (distributional, and hence also strong) solution of 
\[
(\Delta+q)u = 0 \text{ in $\Om$.}
\]
We extend $u$ and $q$ by zero to $\mR^n$ to obtain a compactly supported solution $u \in W^{2,r}(\mR^n)$ to $(\Delta+q)u=0$. Consequently, $u \equiv 0$ by unique continuation (see \cite[Theorem 6.3 and Remark 6.7]{JerisonKenig1985}), which contradicts the fact that $\norm{u}_{H^1(\Om)} = \lim \norm{u_m}_{H^1(\Om)} = 1$.
\end{proof}

The following lemma gives a uniqueness result and estimate for solutions of the semilinear equation, assuming an a priori bound for the $W^{2,r}$ norm.

%For $n \in \{2,3\}$ we include the additional assumption $r \geq 2$, which will also appear in Lemma \ref{lemma_projection_operator}.

\begin{Lemma} \label{lem_carleman_wtwor}
Let $a \in L^r(\Om, C^{1,\alpha}(\mR))$ where $r > n/2$, and let $u_0 \in W^{2,r}(\Om)$ solve $\Delta u_0 + a(x,u_0) = 0$ in $\Om$. If  $u \in W^{2,r}(\Om)$ is any other solution of $\Delta u + a(x,u) = 0$ in $\Om$ and $\norm{u}_{W^{2,r}(\Om)}$, $\norm{u_0}_{W^{2,r}(\Om)} \leq M$, then 
\begin{equation}\label{estimate_lem_carleman}
\norm{u-u_0}_{W^{2,r}(\Om)} \leq C(M,a) (\norm{u-u_0}_{W^{2-1/r,r}(\p\Om)} + \norm{\p_{\nu}(u-u_0)}_{W^{1-1/r,r}(\p\Om)}).
\end{equation}
\end{Lemma}

\begin{proof}
Let $v=u-u_0$. Then 
\begin{equation} \label{deltav_equat}
-\Delta v = a(x,u) - a(x,u_0) = \left[ \int_0^1 \p_u a(x,(1-t)u_0 + t u) \,dt \right] v.
\end{equation}
Writing $q = \int_0^1 \p_u a(x,(1-t)u_0 + t u) \,dt$ and using that $|u|, |u_0| \leq CM$ by Sobolev embedding, we get that 
\[
\norm{q}_{L^r} \leq \int_0^1 \norm{\p_u a(x,(1-t)u_0 + tu)}_{L^r} \,dt \leq \int_0^1 \left[ \int_{\Om} \sup_{|s| \leq C M} |\p_u a(x,s)|^r \,dx \right]^{1/r} \,dt \leq C(M,a).
\]
Since $v \in W^{2,r}(\Om)$ solves $\Delta v + qv = 0$, Lemma \ref{lemma_linear_cd_estimate} yields the required result.
\end{proof}

The next technical lemma, which establishes the existence of a bounded projection operator, will be needed for applying the implicit function theorem when proving that $T_{a_2,w_1}$ is $C^1$. This is the only place where we need the additional assumption $r \geq 2$ for $n \in \{2, 3 \}$.

\begin{Lemma} \label{lemma_projection_operator}
Let $q \in L^r(\Om)$ where $r > n/2$ and $r \geq 2$. Define the spaces 
\[
Y = W^{2,r}_0(\Om), \qquad Z = (\Delta+q)(Y),
\]
where $Z$ is equipped with the $L^r(\Om)$ topology. 
Then $Y$ and $Z$ are Banach spaces and $\Delta+q: Y \to Z$ is an isomorphism. Moreover, there is a bounded linear operator 
\[
P: L^r(\Om) \to Z
\]
such that $P(z) = z$ for all $z \in Z$. It is given by $P(u) = (\Delta+q)y$ where $y \in W^{2,r}_0(\Om)$ is the unique solution of 
\[
\begin{cases}
(\Delta+q)^2 y &= (\Delta + q)u \text{ in $\Om$}, \\
y|_{\p \Om} = \p_{\nu} y|_{\p \Om} &= 0.
\end{cases}
\]
\end{Lemma}
\begin{proof}
Note that $Y$ is a closed subspace of $W^{2,r}(\Om)$ and $Z$ is a closed subspace of $L^r(\Om)$ by Lemma \ref{lemma_linear_cd_estimate}, so both spaces are Banach spaces. Lemma \ref{lemma_linear_cd_estimate} also implies that $\Delta+q: Y \to Z$ is injective, and by definition it is surjective. The open mapping theorem ensures that this map is an isomorphism.

We now assume that $n \geq 4$ and $r > n/2$ (the case $n \in \{ 2, 3 \}$ and $r \geq 2$ follows by straightforward modifications). We first consider the solvability of the fourth order equation in $H^2_0(\Om)$. Note that the bilinear form 
\[
B(y,w) = ((\Delta+q)y, (\Delta+q)w)_{L^2(\Om)}, \qquad y,w \in H^2_0(\Om),
\]
is well defined when $q \in L^r(\Om)$ since $L^r H^2  \subseteq  L^{\frac{n}{2}+} L^{\frac{2n}{n-4}}  \subseteq  L^{2+}$ using the assumption that $n \geq 4$, with a small modification if $n=4$. It satisfies 
\[
B(y,y)^{1/2} \geq \norm{\Delta y}_{L^2} - \norm{q}_{L^r} \norm{y}_{L^{\frac{2r}{r-2}}}.
\]
The interpolation inequality $\norm{y}_{L^{\frac{2r}{r-2}}} \leq \norm{y}_{L^2}^{\theta} \norm{y}_{L^{\frac{2n}{n-4}}}^{1-\theta} \leq C \norm{y}_{L^2}^{\theta} \norm{y}_{H^2}^{1-\theta}$ for some $\theta \in (0,1)$, with a small modification if $n=4$, and Young's inequality imply that for any $\eps > 0$ there is $C_{\eps} > 0$ with 
\[
B(y,y)^{1/2} \geq \norm{\Delta y}_{L^2} - \eps \norm{y}_{H^2} - C_{\eps} \norm{y}_{L^2}.
\]
Now $\norm{\Delta y}_{L^2} \geq c \norm{y}_{H^2}$ for some $c > 0$ since $y \in H^2_0$, so by choosing $\eps$ small enough we obtain 
\[
B(y,y) \geq c \norm{y}_{H^2}^2 - C \norm{y}_{L^2}^2.
\]
Thus the bilinear form $B(y,w) + C(y,w)_{L^2(\Om)}$ is positive definite on $H^2_0(\Om)$. The Riesz representation theorem implies unique solvability of the problem 
\[
(\Delta+q)^2 y + Cy = F \text{ in $\Om$}, \qquad y \in H^2_0(\Om),
\]
for any $F \in H^{-2}(\Om)$. The spectral theorem applied to the solution operator shows that there is a countable sequence of eigenvalues. Now any $y \in H^2_0(\Om)$ with $B(y,y) = 0$ satisfies $(\Delta+q)y = 0$, so $y=0$ by unique continuation. This implies that $0$ is not an eigenvalue of the equation 
\[
(\Delta+q)^2 y = F \text{ in $\Om$}, \qquad y \in H^2_0(\Om).
\]
Therefore this equation has a unique solution $y \in H^2_0(\Om)$ for any $F \in H^{-2}(\Om)$.

Next suppose that $F \in W^{-2,r}(\Om)$ where $r \geq 2$. Then the solution $y \in H^2_0(\Om)$ satisfies 
\begin{equation} \label{fourth_order_eq_bootstrap}
\Delta^2 y = F - q \Delta y - \Delta(qy) - q^2 y
\end{equation}
with zero Cauchy data. Here $F \in L^r(\Om) \subseteq L^2(\Om)$, which belongs to $W^{-2,2+\delta}(\Om)$ for some $\delta > 0$ by Sobolev embedding. Similarly, since $q \in L^r$ for $r > n/2$ and $y \in H^2$, the expression $q(\Delta y)$ is in $L^p$ where $p = \frac{2r}{r+2} > \frac{2n}{n+4}$. Thus $q(\Delta y)$ can be integrated against functions in $W^{2,2-\eps}_0 \subset L^{\frac{2n}{n-4} - \tilde{\eps}}$ for some $\eps, \tilde{\eps} > 0$, so $q(\Delta y)$ is in $W^{-2,2+\delta}$ for some $\delta > 0$. By similar arguments we see that the right hand side of \eqref{fourth_order_eq_bootstrap} is in $W^{-2,2+\delta}(\Om)$ for suitable $\delta > 0$ depending on $n$. Since the equation has smooth coefficients, the solution must be in $W^{2,2+\delta}(\Om)$, see \cite[Theorem 2.22]{Gazzola_book} or \cite{browder62}. Bootstrapping this regularity argument finitely many times shows that $y \in W^{2,r}(\Om)$. We omit the details.
\end{proof}

We can now begin the

\begin{proof}[Proof of Proposition \ref{prop_second_solution_map}]
Let $u_{1,v} = S_{a_1,w_1}(v), u_{2,v} = T_{a_2,w_1}(v)$ and $r_v = u_{1,v}-u_{2,v}$ where $v \in V_{q_1,\delta_1}$. Here $u_{1,v}$ is $C^1$ with respect to $v$, and we wish to show that $r_v$ is also $C^1$ in $v$. Note that $r_v$ solves 
\[
(\Delta + q_2) r_v = q_2 r_v + a_2(x,u_{1,v}-r_v)-a_1(x,u_{1,v}).
\]
Since $r_v \in W^{2,r}_0(\Om)$, denoting by $P$ the projection operator and $Z$ the space in Lemma \ref{lemma_projection_operator} for $q=q_2$, we also have  
\begin{equation} \label{rv_proj_eq}
(\Delta + q_2) r_v = P(q_2 r_v + a_2(x,u_{1,v}-r_v)-a_1(x,u_{1,v})).
\end{equation}
To show $C^1$ dependence in $v$, we define the map 
\[
F: W^{2,r}_0(\Om) \times V_{q_1,\delta_1} \to Z, \ \ F(r,v) = (\Delta + q_2) r - P(q_2 r + a_2(x,u_{1,v}-r)-a_1(x,u_{1,v})).
\]
The projection operator $P$ ensures that $F$ indeed maps into $Z$. The map $F$ is $C^1$ by the assumptions on $a_j$ and $u_{1,v}$, one has $F(r_0,0) = 0$ by \eqref{rv_proj_eq}, and 
\[
(D_r F)_{(r_0,0)}\tilde{r} = (\Delta+q_2) \tilde{r} - P(q_2 \tilde{r} - \p_u a_2(x, w_2) \tilde{r}) = (\Delta+q_2) \tilde{r}.
\]
By Lemma \ref{lemma_projection_operator} the map $(D_r F)_{(r_0,0)} = \Delta + q_2$ is an isomorphism $W^{2,r}_0(\Om) \to Z$.

The implicit function theorem ensures that there is a $C^1$ map $R: V_{q_1,\tilde{\delta}_1} \to W^{2,r}_0(\Om)$ for some $\tilde{\delta}_1 > 0$ such that $F(R(v),v) = 0$ for $v$ near $0$ and 
\[
F(r,v) = 0 \text{ for $(r,v)$ near $(r_0,0)$} \quad \Longleftrightarrow \quad r = R(v).
\]
We also have $F(r_v,v) = 0$ for $v \in V_{q_1,\delta_1}$ by \eqref{rv_proj_eq}, and we would like to show that $r_v$ is close to $r_0 = w_1-w_2$. This would yield $r_v = R(v)$ by the uniqueness statement above. Now the properties of $S_{a_1,w_1}$ imply 
\[
\norm{u_{1,v}-w_1}_{W^{2,r}} \leq 2 \norm{v}_{W^{2,r}}
\]
for $v$ small enough. For $u_{2,v}$ we use Lemma \ref{lem_carleman_wtwor} with $u = u_{2,v}$, which satisfies $\norm{u_{2,v}}_{L^{\infty}} \leq C \norm{u_{2,v}}_{W^{2,r}} \leq C$ by the assumption $C_{a_1}^{w_1, \delta} \subseteq C_{a_2}^{0,C}$, and $u_0 = w_2$ to obtain 
\begin{align*}
\norm{u_{2,v}-w_2}_{W^{2,r}} &\leq C (\norm{u_{1,v}-w_1}_{W^{2-1/r,r}(\p\Om)} + \norm{\p_{\nu}(u_{1,v}-w_1)}_{W^{1-1/r,r}(\p\Om)}) \leq C \norm{u_{1,v}-w_1}_{W^{2,r}} \\
 &\leq C \norm{v}_{W^{2,r}}.
\end{align*}
Here we used that $u_{2,v}$ and $w_2$ have the same Cauchy data as $u_{1,v}$ and $w_1$, respectively. Combining these two estimates gives 
\[
\norm{r_v-r_0}_{W^{2,r}} \leq \norm{u_{1,v}-w_1}_{W^{2,r}} + \norm{u_{2,v}-w_2}_{W^{2,r}} \leq C \norm{v}_{W^{2,r}}.
\]
Thus $r_v$ is close to $r_0$ when $v$ is small as required, so $r_v = R(v)$ depends in a $C^1$ way on $v$.
\end{proof}

\section{Inverse problem}\label{sec_proof_of_main}

In this section we prove Theorems \ref{thm_main1} and \ref{thm_main2}. The proof is almost the same as the proofs of the main results in \cite{JNS2023} and thus we do not give the full details here.

Recall that we have $a_1, a_2 \in L^r(\Om, C^{1,\alpha}(\mR))$, $r> n/2$ and $r\geq 2$, and $w_1\in W^{2,r}(\Om)$ solves $\Delta w_1 + a_1(x,w_1)=0$ in $\Om$. We further assume that for some $\delta, C > 0$ we have $C^{w,\delta}_{a_1}\subseteq C^{0,C}_{a_2}$ and define
\begin{align*}
    V_q&=\{ v\in W^{2,r}(\Om) : \Delta v + qv = 0\text{ in } \Om\}\\
    V_{q,\delta} &= \{ v\in V_q : \norm{v}_{W^{2,r}(\Om)}<\delta\},
\end{align*}
where $q=\p_ua_1(x,w_1)$. Now for any $v\in V_{q,\delta}$, $\delta$ small, Propositions \ref{prop_first_solution_map} and \ref{prop_second_solution_map} give the solutions $u_{1,v}=S_{a_1,w_1}(v)$ and $u_{2,v}=T_{a_2,w_1}(v)$ solving $\Delta u_{j,v} + a_j(x,u_{j,v})=0$ in $\Om$. We can then prove the following lemmas (corresponding to \cite[Lemma 5.1 and Lemma 5.2]{JNS2023}):

\begin{Lemma}\label{first_lin_ident}
    Assume $C^{w,\delta}_{a_1}\subseteq C^{0,C}_{a_2}$. Then there is $\delta_1>0$ such that for any $v\in V_{q,\delta_1}$ one has
    \begin{equation*}
        \p_ua_1(x,u_{1,v}(x))=\p_ua_2(x,u_{2,v}(x)), \quad\text{for }x\in\Om.
    \end{equation*}
\end{Lemma}
\begin{proof}
    The proof is based on the fact that the solution operators are $C^1$ and that the derivative $DS(v)$ is an isomorphism (see Lemma \ref{S_deriv_isom}). As in \cite[Lemma 5.1]{JNS2023}, using these and the assumption $C^{w,\delta}_{a_1}\subseteq C^{0,C}_{a_2}$ we get an integral identity
    \begin{equation*}
        \int_{\Om} \big(\p_ua_1(x,u_{1,v}(x)) - \p_ua_2(x,u_{2,v}(x))\big)v_1v_2\,dx=0
    \end{equation*}
    for any solutions $v_j$ of $(\Delta + \p_ua_j(x,u_{j,v}))v_j=0$. Using the completeness of products for such solutions (\cite{Chanillo1990, Nachman1992, DKS2013} when $n\geq3$, and \cite{Blåsten2020} for $n=2$) gives the result.
\end{proof}

\begin{Remark}
The results in \cite{Chanillo1990, Nachman1992, DKS2013, Blåsten2020} are stated in terms of DN maps, but they indeed prove the following completeness statement: if
\[ 
\int_{\Omega} f u_1 u_2 \,dx = 0
\]
for all $u_j$ solving $(\Delta+q_j)u_j = 0$ in $\Omega$, then $f=0$ (with $f \in L^{n/2}(\Om)$ for $n \geq 3$, and $f \in L^2(\Omega)$ for $n=2$).

For the case $n \geq 3$, see \cite[argument after (4.1) in proof of Theorem 1.1]{DKS2013} where $q = q_1-q_2$ can be replaced by a general function $f \in L^{n/2}$. For $n=2$ one uses the argument in \cite{Blåsten2020} instead. The point is that the lemmas in \cite[Section 5]{Blåsten2020} go through with the same proofs when $q_1-q_2$ is replaced by a general function $f$ with the same properties. The only caveat is the use of reference [12] in the beginning of \cite[proof of Theorem 2.1]{Blåsten2020}, which shows that $\Lambda_{q_1} = \Lambda_{q_2}$ implies $q_1-q_2 \in L^2(\Omega)$. However, we have already assumed $q_j \in L^2(\Omega)$ when $n=2,3$, so we do not need this step at all.
\end{Remark}

\begin{Lemma}\label{phi_ind_v}
    In the setting of Lemma \ref{first_lin_ident} the function $\varphi_v= u_{2,v}-u_{1,v}$ does not depend on $v\in V_{q,\delta_1}$.
\end{Lemma}
\begin{proof}
    The function $\psi_{t}=\varphi_{tv}$ is $C^1$ in $t$, has zero Cauchy data and $z_t=\p_t\psi_t$ satisfies (using Lemma \ref{first_lin_ident})
    \begin{equation*}
        \Delta z_t + \p_ua(x,u_{1,tv})z_t=0.
    \end{equation*}
    Furthermore $z_t$ has zero Cauchy data, and hence by unique continuation (or by Lemma \ref{lemma_linear_cd_estimate}) $z_t = 0$. Thus $\psi_t$ does not depend on $t$, and $\varphi_v=\varphi_0$.
\end{proof}

The proof of our main result goes in the same way as in \cite[Section 5]{JNS2023}.

\begin{proof}[Proof of Theorem \ref{thm_main2}]
    Using Lemma \ref{phi_ind_v} we have 
    \begin{equation*}
        \Delta \varphi = \Delta (u_{2,v} - u_{1,v}) = a_1(x,u_{1,v}) - a_2(x,u_{1,v} + \varphi)
    \end{equation*}
    which implies $a_1(x,u_{1,v}(x))=(T_{\varphi}a_2)(x,u_{1,v}(x))$. It remains to show that there is $\eps > 0$ such that for any $\bar{x} \in \ol{\Om}$ and any $\lambda \in (-\eps,\eps)$ there is a small solution $v$ with $u_{1,v}(\bar{x}) = w(\bar{x}) + \lambda$.

    Recall from Proposition \ref{prop_first_solution_map} that there is a function $\delta(t)$ with $\delta(t) \to 0$ as $t \to 0$ such that 
    \[
    u_{1,v} = w_1 + v + R_v
    \]
    where $\norm{R_v}_{W^{2,r}} \leq \norm{v}_{W^{2,r}} \delta(\norm{v}_{W^{2,r}})$ for $v$ small.
    If $x_0 \in \ol{\Om}$, we use Runge approximation (Proposition \ref{prop_runge_nonvanishing_solution}) to find a solution $v = v_{x_0}$ with $v(x_0) = 4$. By continuity, $v \geq 2$ in $\ol{U}_{x_0} \cap \ol{\Om}$ for some neighborhood $U_{x_0}$ of $x_0$. Now for $x \in \ol{U}_{x_0} \cap \ol{\Om}$ and $t$ small, we have 
    \[
    |u_{1,tv}(x) - w_1(x)| \geq |tv(x)| - |R_{tv}(x)| \geq 2|t| - C \norm{R_{tv}}_{W^{2,r}} \geq 2|t| - C \norm{tv}_{W^{2,r}} \delta(\norm{tv}_{W^{2,r}}).
    \]
    Thus there is $\eps_{x_0}$ such that for $|t| \leq \eps_{x_0}$, one has 
    \[
    |u_{1,tv}(x) - w_1(x)| \geq |t|, \qquad x \in \ol{U}_{x_0} \cap \ol{\Om}.
    \]

    Next we use compactness of $\ol{\Om}$ to find points $x_1, \ldots, x_N \in \ol{\Om}$ such that $\ol{\Om} \subseteq U_{x_1} \cup \ldots \cup U_{x_N}$. Choose $\eps = \min\{\eps_{x_1}, \ldots, \eps_{x_N} \}$. Given $\bar{x} \in \ol{\Om}$, let $j$ be such that $\bar{x} \in U_{x_j}$ and define 
    \[
    \eta(t) = u_{1,tv_{x_j}}(\bar{x}) - w_1(\bar{x}).
    \]
    Then $\eta(\eps) \geq \eps$ and $\eta(-\eps) \leq -\eps$. By continuity, for any $\lambda \in (-\eps,\eps)$ there is $t$ with $\eta(t) = \lambda$. This ends the proof.
\end{proof}

Theorem \ref{thm_main1} is a direct consequence of Theorem \ref{thm_main2}.

\begin{comment}
The proof of our main result goes the same way as in \cite[Section 5]{JNS2023}. The first main difference is that we use the completeness of products of solutions result from \cite{Chanillo1990, Nachman1992} when $n\geq3$, and from \cite{Blåsten2020} for $n=2$. The first two references contain the completeness result for potentials in $L^{n/2}$ when $n \geq 3$, and the last one for potentials in $L^{4/3+\eps}$ when $n=2$. The second difference is that we use a Runge approximation result valid for $L^r$ potentials. This is proved in Section \ref{sec_runge}.

{\color{red} Should follow as in earlier paper. The main difference is that the Runge approximation result is for $L^r$ potentials and the natural approximation would be in the $W^{2,r}$ or slightly weaker norm.}
\end{comment}

\section{Runge approximation} \label{sec_runge}

The following result was needed in the proofs of the main theorems.

\begin{Proposition} \label{prop_runge_nonvanishing_solution}
Let $\Omega \subseteq \mR^n$ be an open set, and let $q \in L^r(\Om)$ where $r > n/2$. Given any $x_0 \in \Om$, there exists $u \in W^{2,r}(\Om)$ solving $(\Delta+q)u = 0$ in $\Om$ such that $u(x_0) \neq 0$.
\end{Proposition}

We begin by constructing a solution that is positive in a small ball.

\begin{Lemma} \label{lemma_nonzero_solution_small_ball}
Let \( q\in L^{r}(B_{1}) \), $r > n/2$.
There is $\eps_0 > 0$ such that for $0 < \eps \leq \eps_0$ there is a unique solution $u\in W^{2,r}(B_{\varepsilon})$ of
\begin{equation}\label{eq-nonzero-1_new}
\begin{cases}
	(\Delta+q)u = 0&\text{in }B_{\varepsilon},\\
	u = \varepsilon^{2}&\text{on }\partial B_{\varepsilon},
\end{cases}
\end{equation}
such that $u$ is positive in $\ol{B}_{\eps}$.
\end{Lemma}
\begin{proof}
A function \( u \) satisfies \eqref{eq-nonzero-1_new} if and only if \( u_{\varepsilon}(x) = \varepsilon^{-2}u(\varepsilon x) \) satisfies, for \( x\in B_{1} \),
\begin{equation*}
	\Delta u_{\varepsilon}(x) = \varepsilon^{-2}\Delta[ u(\varepsilon x)] = [\Delta u(y)]\vert_{y=\varepsilon x} = -q_{\varepsilon}(x)u(\varepsilon x) = -\varepsilon^{2} q_{\varepsilon}(x) u_{\varepsilon}(x),
\end{equation*}
where \( q_{\varepsilon}(x) = q(\varepsilon x) \in L^r(B_1) \).
So \( u_{\varepsilon} \) should satisfy the equation
\begin{equation*}
\begin{cases}
	(\Delta +\varepsilon^{2}q_{\varepsilon})u_{\varepsilon} = 0 &\text{in } B_{1},\\
	u_{\varepsilon}=1&\text{on }\partial B_{1}.
\end{cases}
\end{equation*}
Taking \( u_{\varepsilon}(x) = 1+r_{\varepsilon}(x) \), we see that $r_{\eps}$ should satisfy the Poisson equation
\begin{equation} \label{eq-nonzero-r_1_new}
\begin{cases}
	(\Delta +\varepsilon^{2}q_{\varepsilon})r_{\varepsilon} = -\eps^2 q_{\eps} &\text{in } B_{1},\\
	r_{\varepsilon} = 0&\text{on }\partial B_{1}.
\end{cases}
\end{equation}

Next we show that \eqref{eq-nonzero-r_1_new} has a unique solution, i.e.\ $N_{\eps^2 q_{\eps}} = \{0\}$ in the notation of Section \ref{sec_solvability_linear}. Note that 
\[
\norm{\eps^2 q_{\eps}}_{L^r(B_1)} = \eps^{2-n/r} \norm{q}_{L^r(B_{\eps})} \leq \eps^{2-n/r} \norm{q}_{L^r(B_{1})}.
\]
Since $r > n/2$, we have $2-n/r > 0$. Now if $\psi \in N_{\eps^2 q_{\eps}}$, i.e.\ $\psi \in W^{2,r}(B_1)$ satisfies 
\begin{equation*}
\begin{cases}
	\Delta \psi = -\varepsilon^{2}q_{\varepsilon} \psi &\text{in } B_{1},\\
	\psi = 0&\text{on }\partial B_{1},
\end{cases}
\end{equation*}
then by Proposition \ref{prop_linearized_ri_wonep} (since $N_{0} = \{0\}$) we have 
\[
\norm{\psi}_{W^{2,r}(B_1)} \leq C_{n,r} \norm{\varepsilon^{2}q_{\varepsilon} \psi}_{L^r(B_1)} \leq C_{n,r,q} \eps^{2-n/r} \norm{\psi}_{L^{\infty}(B_1)} \leq C_{n,r,q} \eps^{2-n/r} \norm{\psi}_{W^{2,r}(B_1)}.
\]
In the last step we used Sobolev embedding. By choosing $\eps \leq \eps_0 = \eps_0(n,r,q)$, we may absorb the right hand side to the left. It follows that $N_{\eps^2 q_{\eps}} = \{0\}$ for $\eps \leq \eps_0$.

Proposition \ref{prop_linearized_ri_wonep} now ensures that \eqref{eq-nonzero-r_1_new} has a unique solution $r_{\eps} \in W^{2,r}(B_1)$. Writing the equation as $\Delta r_{\eps} = -\varepsilon^{2}q_{\varepsilon} r_{\varepsilon} - \eps^2 q_{\eps}$ and using the norm estimate in Proposition \ref{prop_linearized_ri_wonep} yields 
\[
\norm{r_{\eps}}_{W^{2,r}(B_1)} \leq C_{n,r} \norm{\eps^2 q_{\eps}}_{L^r(B_1)} (\norm{r_{\eps}}_{L^{\infty}(B_1)} + 1).
\]
We may use the estimate $\norm{\eps^2 q_{\eps}}_{L^r(B_1)} \leq \eps^{2-n/r} \norm{q}_{L^r(B_{1})}$ above and Sobolev embedding to conclude that 
\[
\norm{r_{\eps}}_{W^{2,r}(B_1)} \leq C_{n,r,q} \eps^{2-n/r}
\]
for $\eps$ sufficiently small. By Sobolev embedding again, we obtain $\norm{r_{\eps}}_{L^{\infty}(B_1)} \leq 1/2$ for $\eps$ small, so $u_{\eps} = 1+r_{\eps} \geq 1/2$ in $B_1$ for $\eps$ small. It follows that $u(y) = \eps^2 u_{\eps}(y/\eps) \geq \eps^2/2$ for $y \in B_{\eps}$.
\end{proof}

The following result proves Runge approximation in our setting. The argument in \cite{Leis1967} proves that Dirichlet eigenvalues are strictly decreasing with respect to increasing the domain also under our assumptions. Combining this with Proposition \ref{prop_linearized_ri_wonep} shows that a ball $B$ as in the statement below can always be found.

\begin{Lemma}[Runge approximation] \label{lemma_runge_wtwor}
Let \( U \subset\mathbb{R}^{n} \) be an open set with smooth boundary such that $\mR^n \setminus \ol{U}$ is connected.
Let \( q\in L^{r}(U) \), \( r>n/2 \), and extend \( q \) by zero into an open ball \( B\subset\mathbb{R}^{n} \) containing \( \ol{U} \) such that the equation
\begin{equation*}
\begin{cases}
	(\Delta+q)u = 0&\text{in }B,\\
	u = f&\text{on }\partial B,
\end{cases}
\end{equation*}
is well-posed in \( W^{2,r}(B) \) for \( f\in W^{2-1/r,r}(\partial B) \).
Define the spaces
\begin{equation*}
\begin{aligned}
	V_{B} &= \{u\in W^{2,r}(B)\colon (\Delta+q)u = 0 \text{ in $B$} \}\\
	V_{U} &= \{u\in W^{2,r}(U)\colon (\Delta+q)u = 0 \text{ in $U$} \} \\
	R_{U} &= \{u\vert_{U} \colon u\in V_{B} \}.
\end{aligned}
\end{equation*}
Then \( R_{U} \) is dense in \( V_{U} \) with respect to the \( W^{2,r}(U) \)-norm.
\end {Lemma}

The proof is based on the Hahn-Banach theorem, unique continuation, and a duality argument involving very weak solutions. Let $q \in L^{r}(\Omega)$, $r> n/2$ and assume that $0$ is not a Dirichlet eigenvalue of $\Delta+q$ in $\Omega$. If $\mu$ is a bounded linear functional on $W^{2,r}(\Om)$ and $\frac{1}{r} + \frac{1}{r'} = 1$, we say that $u \in L^{r'}(\Omega)$ is a \emph{very weak solution} of
%\begin{equation} \label{veryweak_def}
\[
(\Delta+q)u = \mu \text{ in $\Omega$}, \qquad u|_{\p \Omega} = 0,
\] %\end{equation}
if 
\[
\int_{\Om} u(\Delta+q)\varphi \,dx = \mu(\varphi)
\]
for any $\varphi \in W^{2,r}(\Om) \cap W^{1,r}_0(\Om)$. The next result will be used in the proof of Lemma \ref{lemma_runge_wtwor}.

\begin{Lemma} \label{lemma_very_weak_sol_lrprime}
Let \( r>n/2 \) and \( \Omega\subset\mathbb{R}^{n} \) be an open set with smooth boundary.
Let \( q\in L^{r}(\Omega) \) and suppose that the equation
\begin{equation*}
\begin{cases}
	(\Delta+q)u = F&\text{in }\Omega,\\
	u = 0&\text{on }\partial \Omega,
\end{cases}
\end{equation*}
is well-posed in \( W^{2,r}(\Omega) \) for \( F\in L^{r}(\Omega) \).
Then there exists for every functional \( \mu\in W^{2,r}(\Omega)^{*} \) a very weak solution \( v\in L^{r'}(\Omega) \) of
\begin{equation*}
\begin{cases}
	(\Delta+q)v = \mu&\text{in }\Omega,\\
	v = 0&\text{on }\partial\Omega.
\end{cases}
\end{equation*}
\end{Lemma}
\begin{proof}
By the well-posedness \( P=\Delta+q \) maps \( W^{2,r}(\Omega)\cap W^{1,r}_{0}(\Omega) \) surjectively onto \( L^{r}(\Omega) \) and its inverse \( P^{-1} \) is bounded.
Define for \( \varphi\in L^{r}(\Omega) \) the functional \( L\varphi = \inner{\mu}{P^{-1}\varphi} \).
Then \( L \) is bounded,
\begin{equation*}
	\abs{L\varphi} = \abs{\inner{\mu}{P^{-1}\varphi}} \leq C \norm{P^{-1}\varphi}_{W^{2,r}(\Omega)}\leq \tilde{C}\norm{\varphi}_{L^{r}(\Omega)}
\end{equation*}
and hence $L$ belongs to \( L^{r}(\Omega)^{*} \).
By Riesz representation theorem, we get \( v\in L^{r'}(\Omega) \) such that
\begin{equation*}
	\int_{\Omega}v\varphi = L\varphi = \inner{\mu}{P^{-1}\varphi}.
\end{equation*}
Using the well-posedness we may write any \( \varphi \) as \( \varphi = P\psi \) for \( \psi\in W^{2,r}(\Omega)\cap W^{1,r}_{0}(\Omega) \) and hence
\begin{equation*}
	\int_{\Omega}vP\psi = \inner{\mu}{\psi}\quad\forall \psi\in W^{2,r}(\Omega)\cap W^{1,r}_{0}(\Omega).\qedhere
\end{equation*}
\end{proof}

\begin{proof}[Proof of Lemma \ref{lemma_runge_wtwor}]
We equip $V_{U}$ with the $W^{2,r}(U)$ norm, so $V_{U}$ is a closed subspace of $W^{2,r}(U)$ and $R_{U}$ is a linear subspace. By the Hahn-Banach theorem, $R_{U}$ will be dense in $V_{U}$ if every bounded linear functional on $W^{2,r}(U)$ that vanishes on $R_{U}$ must also vanish on $V_{U}$. Let \( \tilde{h}\in W^{2,r}(U)^{*} \) such that
\begin{equation}\label{eq-runge2-5}
	\inner{\tilde{h}}{v} = 0\quad\forall v\in R_{U}.
\end{equation}
We wish to show that $\inner{\tilde{h}}{v} = 0$ for all $v \in V_{U}$.

Let \( h\in W^{2,r}(B)^{*} \) be the extension of \( \tilde{h} \) defined by \( \inner{h}{\varphi}\coloneqq \inner{\tilde{h}}{\varphi\vert_{U}} \), for \( \varphi \in W^{2,r}(B) \).
Then \( \operatorname{supp}(h)\subseteq \ol{U} \).
We use Lemma \ref{lemma_very_weak_sol_lrprime} and let \( w\in L^{r'}(B) \) be a very weak solution of the adjoint problem
\begin{equation*}
\begin{cases}
	(\Delta+q)w = h&\text{in }B,\\
	w = 0&\text{on }\partial B.
\end{cases}
\end{equation*}
In other words, \( w \) satisfies
\begin{equation}\label{eq-runge2-1}
	\int_{B}w[\Delta+q]\varphi\,dx = \inner{h}{\varphi}\quad\forall\varphi\in W^{2,r}(B)\cap W^{1,r}_{0}(B).
\end{equation}
There is a bounded extension operator \( W^{2,r}(U) \to W^{2,r}_{0}(B) \).
We may therefore extend any solution \( v\in V_{U} \) to a function \( \tilde{v}\in W^{2,r}_{0}(B) \).
Taking \( \varphi = \tilde{v} \) in \eqref{eq-runge2-1}, we have
\begin{equation}\label{eq-runge2-2}
	\int_{B\setminus U}w[\Delta+q]\tilde{v}\,dx = \inner{h}{\tilde{v}}.
\end{equation}
The main step is to show that \( w = 0 \) in \( B\setminus \ol{U} \).
This will be accomplished by unique continuation from an open set.

Let \( Q\subset\mathbb{R}^{n} \) be a bounded open set containing \( B \) and define \( u\in L^{r'}(Q) \) by
\begin{equation*}
	u(x)=
\begin{cases}
	w(x)&\text{if }x\in B, \\
	0&\text{otherwise}.
\end{cases}
\end{equation*}
We show that \( u \) is a distributional solution of \( \Delta u = 0 \) in \( Q\setminus\overline{U} \) in the sense that
\begin{equation*}
	\int_{Q}u \Delta\varphi \,dx = 0\quad\forall \varphi\in C^{\infty}_{c}(Q\setminus\overline{U}).
\end{equation*}
For such \( \varphi \), extend it by zero into \( U \), so that \( \varphi\in C^{\infty}_{c}(Q) \) with \( \operatorname{supp}(\varphi)\cap \ol{U} = \emptyset \).
Let \( \psi\in W^{2,r}(B) \) satisfy
\begin{equation*}
\begin{cases}
	(\Delta+q) \psi = 0&\text{in }B,\\
	\psi = \varphi&\text{on }\partial B.
\end{cases}
\end{equation*}
Then \( \pi = \varphi\vert_{B}-\psi\in W^{2,r}(B)\cap W^{1,r}_{0}(B) \).
From this and \( u\vert_{Q\setminus B} = 0 \), \( q|_{B \setminus U} = 0\), and \( \varphi|_{U} = 0 \) we have
\begin{equation*}
\begin{aligned}
	\int_{Q}u \Delta\varphi \,dx &= \int_{B}w \Delta \varphi \,dx = \int_{B}w (\Delta+q) \varphi \,dx = \int_{B}w (\Delta+q) \pi \,dx \\
	&= \inner{h}{\pi} = \inner{h}{\varphi\vert_{B}} - \inner{h}{\psi}.
\end{aligned}
\end{equation*}
Now \( \psi\in V_{B} \) so we have \( \inner{h}{\psi} = 0 \) by \eqref{eq-runge2-5}.
And \( \inner{h}{\varphi\vert_{B}} = 0 \) follows from \( \operatorname{supp}(h)\cap \operatorname{supp}(\varphi\vert_{B}) = \emptyset \).
This shows that \( u \) is a distributional solution of $\Delta u = 0$ in \( Q\setminus\overline{U} \). Since $u = 0$ in $Q \setminus B$, unique continuation (see \cite[Theorem 6.3 and Remark 6.7]{JerisonKenig1985}) shows that \( w = u = 0 \) in \( B\setminus\overline{U} \). Here we used that $\mR^n \setminus \ol{U}$ is connected. 
Now \eqref{eq-runge2-2} reduces to
\begin{equation*}
	\inner{h}{v} = 0\quad\forall v\in V_{U}
\end{equation*}
which completes the proof.
\end{proof}

\begin{proof}[Proof of Proposition \ref{prop_runge_nonvanishing_solution}]
We begin by using Lemma \ref{lemma_nonzero_solution_small_ball} to find a positive function $u_0 \in W^{2,r}(B(x_0,\eps))$ solving $(\Delta+q)u_0 = 0$ in $B(x_0, \eps)$, where $\eps > 0$ is sufficiently small. Then Lemma \ref{lemma_runge_wtwor}, with $U = B(x_0,\eps)$ and $B$ a suitable large ball containing $\ol{\Om}$, ensures that there are $u_j \in W^{2,r}(B)$ solving $(\Delta+q)u_j = 0$ in $B$ such that $u_j|_{B(x_0,\eps)} \to u_0$ in $W^{2,r}(B(x_0,\eps))$ and thus also in $L^{\infty}(B(x_0,\eps))$ by Sobolev embedding. By choosing $j$ large enough we see that $u_j|_{B(x_0,\eps)}$ is positive, which proves the result.
\end{proof}

\printbibliography

\end{document}